\documentclass[11pt, twoside, letter]{amsart}
\usepackage[]{appendix}

\usepackage{matthew}
\usepackage{mathrsfs}

\usepackage[pdftex,plainpages=false,hypertexnames=false,pdfpagelabels]{hyperref}
\newcommand{\arXiv}[1]{\href{http://arxiv.org/abs/#1}{\tt arXiv:\nolinkurl{#1}}}
\newcommand{\arxiv}[1]{\href{http://arxiv.org/abs/#1}{\tt arXiv:\nolinkurl{#1}}}
\newcommand{\doi}[1]{\href{http://dx.doi.org/#1}{{\tt DOI:#1}}}

\newcommand{\mathscinet}[1]{\href{http://www.ams.org/mathscinet-getitem?mr=#1}{\tt #1}}
\newcommand{\googlebooks}[1]{(preview at \href{http://books.google.com/books?id=#1}{google books})}
\renewcommand{\MR}[1]{}
\usepackage{hyperref}

\newcommand{\id}{{\mathbf 1}}

\allowdisplaybreaks[4]

\title{Classification of Metaplectic Fusion Categories}

\author{Eddy Ardonne}
\email{ardonne@fysik.su.se}
\address{Department of Physics\\Stockholm University
    \\Albanova University Center
    \\SE-106 91 Stockholm
    Sweden}
\author{Peter E. Finch}
\email{pfinch.mathphys@gmail.com}
\address{ Institut f{\"u}r Theoretische Physik \\
   Leibniz Universit{\"a}t Hannover \\
   Appelstra{\ss}e 2\\
   30167 Hannover\\
   Germany}
\author{Matthew Titsworth}
\email{matthew.titsworth@gmail.com}
\address{Department of Physics, the University of Texas at Dallas, USA}

\begin{document}
\maketitle
\begin{abstract}
In this paper, we study a family of fusion and modular systems realizing fusion categories Grothendieck equivalent to the representation category for $\sost{2p+1}$.
These categories describe non-abelian anyons dubbed `metaplectic anyons'.
We obtain explicit expressions for all the $F$- and $R$-symbols.
Based on these, we conjecture a classification for their monoidal equivalence classes from an analysis of their gauge invariants and define a function which gives us the number of classes.
\end{abstract}
\section{Introduction}\label{s:intro}

Landau \cite{Landau1937} showed that most phases of matter can be classified by investigating in which way the underlying symmetry of the system can be broken. There are, however, also phases of matter that defy this classification, in particular, topologically ordered phases \cite{wen95}, the most famous example being the fractional quantum Hall systems \cite{PhysRevLett.48.1559,PhysRevLett.50.1395}.

It is thus natural to try to classify topologically ordered phases, which can be described using the theory of modular tensor categories. Classifying all modular tensor categories is extremely difficult, with only partial results available, see for instance \cite{MR2544735}.

In this paper, we study a family of fusion and modular systems that go under the name of `metaplectic anyons'. Metaplectic anyon systems can be viewed as a generalization of the `Ising anyon' system, and were studied, for instance, in \cite{MR3215577,PhysRevB.87.165421}. Thus, we fix the fusion rules, and given these fusion rules, we try to find all possible, inequivalent physical solutions. In the process, we obtain explicit expressions for the $F$- and $R$-symbols, which is useful, for instance when one uses these metaplectic anyons to construct spin-chain like models as in \cite{Finch_2018}.
Because we use quite a bit of tensor category machinery, we are forced to take a mathematical point of view on the problem.

Fusion categories over a field $k$ generalize categories of finite group representations; they are tensor categories with finitely many classes of simple objects, all having duals.
They arise in representation theory\cite{MR1321145}, operator algebras\cite{MR1642584}, topological quantum field theories\cite{MR1797619}, and quantum invariants of $3$-manifolds\cite{MR2654259}.
They also play an important role physically in the study of topological quantum computing\cite{MR2640343} and topological phases of matter\cite{PhysRevB.71.045110}.

As such, a classification of fusion categories up to various notions of equivalence is desirable.
One such notion of equivalence is {\em Grothendieck equivalence}, wherein categories have isomorphic decategorifications, which is a based ring in the sense of \cite{MR1976459}.
Another notion of equivalence is {\em monoidal equivalence}, wherein categories are related via an invertible monoidal functor.

A natural line of inquiry is then: Given a based ring, does it admit categorification? 
If so, given its Grothendieck class of fusion categories, how many monoidal classes are there and how can they be distinguished?
This is an incredibly difficult question to answer in general, especially at the level of fusion categories, with no further assumptions.
Most results along these lines are either obtained for small examples which can be easily computed\cite{MR1981895,MR3427429,MR2559711}, but there are some families\cite{MR1659954,MR1237835} for which classifications are known.

In this paper we obtain an answer to these questions for a family of categories we will call metaplectic fusion categories. 
These are fusion categories underlying metaplectic modular categories.
Recall that a metaplectic modular category\cite{MR3215577,PhysRevB.87.165421} is a modular category Grothendieck equivalent to $\sost{2p+1}$, the category of affine $\mathfrak{so}(2p+1)$ representations with highest integer weight $2$\cite{MR1887583}.
Other common constructions for such categories come from $\mathcal C(B_p,2(2p+1),q)$\cite{MR2263097} for $q$ some $2(2p+1)$ root of unity, or from $\Z_2$ equivariantizations of $\mathcal{TY}(\Z_N,\chi, \tau)$\cite{MR2587410} where $N=2p+1$.
In \cite[Theorem 3.2]{2016arXiv160105460A} it was shown that all metaplectic modular categories arise from gauging the particle-hole symmetry in $\Z_{2p+1}$ modular categories.

Our main result is obtained by classifying solutions to the pentagon equations and hexagon equations coming from the $\sost{2p+1}$ based ring.
We stop short of demonstrating that these are all of the categories as this is a much harder problem.
However, the modular data for $\sost{2p+1}$ categories is known and the modular data computed for our solutions coincides.
As such, we refine the classification of \cite{2016arXiv160105460A} one step further.

For fixed $p$, our categories are parameterized by pairs $(r,\kappa)$ where $\kappa=\pm 1$ and $r$ a positive odd integer less than and co-prime to $2p+1$.
Let $R$ be the set of these $r$. 
Define $\mathbb G_{2p+1}^\times = \Z_{2p+1}^\times/\langle 1,-1\rangle$ with $g:\Z_{2p+1}^\times \rightarrow \mathbb G_{2p+1}^\times$ the quotient map.
Elements of $\Z_{2p+1}^\times$ can be represented as positive integers less than and co-prime to $2p+1$. 
Likewise, elements of $\mathbb G_{2p+1}^\times$ can be represented as positive integers less than or equal to $p$ and co-prime to $2p+1$.
Particularly, each $r\in R$ represents an element of $\Z_{2p+1}^\times$ and the action of $\mathbb G_{2p+1}^\times$ on $R$ via $z \cdot R = \{r z^2 \mod (2p+1) \vert r \in R\}$ makes sense.
Similarly, the evaluation $g(r)\vert r \in R$ makes sense.
Our main result can then be stated as follows
\begin{theorem}
For fixed $p$ the following are true:
\begin{enumerate}
	\item The monoidal classes of fusion categories constructed from $(r,\kappa)$, $\kappa=\pm 1$ and $r\in R$ are the fusion categories underlying metaplectic modular categories.
	\item Let $(r,\kappa)$ and $(r',\kappa')$ parameterize two different solutions to the pentagon equations.
		  Then the fusion categories constructed from these solutions are monoidally equivalent if and only if $\kappa = \kappa'$ and there exists $z\in \mathbb G_{2p+1}^\times$ such that $g(r')=g(r z^2)$.
	\item For $2p+1= p_{1}^{a_1}\ldots p_{l}^{a_l}$, there are exactly $2^{l+1}$ monoidally inequivalent metaplectic modular categories if $\exists b \in Z_{2p+1}^\times \vert b^2 = -1$, otherwise there are exactly $2^{l}$.
\end{enumerate}
\end{theorem}
\noindent
As a sanity check on this, following \cite[Theorem 3.2]{2016arXiv160105460A}, we note that the number of metaplectic fusion categories admitting the structure of a modular category is always less than or equal to the number of metaplectic modular categories\footnote{There could, in principle, be fusion categories which do not admit a modular structure.}.

The structure of this paper is as follows.
In \ref{s:fusion} we review the basic information for fusion and modular categories and their arithmetic descriptions, fusion and modular systems.
In \ref{s:invs} we review the gauge invariants of \cite{2015arXiv150903275H} which allow us to extend deductions about our fusion systems to entire gauge and monoidal classes of fusion categories.
In \ref{s:sost} we present our solutions to the pentagon and hexagon equations, and construct the modular data.
In \ref{s:equivs} we determine the monoidal equivalence classes of the fusion categories.
In \ref{s:examples} we explicitly compute the equivalence classes and modular structures for several examples.
Appendix \ref{a:proof} demonstrates that the explicit $F$- and $R$-symbols we present are indeed solutions to the pentagon and hexagon equations.

\noindent {\em Acknowledgements.}
E.~Ardonne would like to thank Eric Rowell, Joost Slingerland and Zhenghan Wang for stimulating discussions.
M.~Titsworth would like to thank the organizers of the 2014 AMS Mathematics Research communities program on Fusion Categories and Quantum information: Zhenghan Wang, Eric Rowell, and Richard Ng.
E.~Ardonne was supported, in part, by the Swedish research council, under Grant No. 2015-05043.
P.~E.~Finch was supported by the Deutsche Forschungsgemeinschaft under Grant No. Fr 737/7-1.

\section{Preliminaries}\label{s:fusion}
In this section we review relevant facts about fusion/modular categories and their arithmetic descriptions.
We do not provide proofs, but most can be found in one of \cite{MR2183279, 2013arXiv1305.2229D, MR1976459, MR1797619}, or now \cite{MR3242743}.
In the sequel we will always work over $k=\C$ for simplicity.

\subsection{Fusion Categories and Modular Categories}
\begin{definition}[{\cite[Definitions 2.1 and 2.2]{MR1976459}}]\label{def:basedring}
A {\em unital based ring} $(R,B)$ is a $\Z^+$- ring $R$ together with a set $B\subset R$ and identity $1_R\in B$ such that
\begin{itemize}
	\item (Structure constants) There exist non-negative integers $N_{XY}^Z$ for $X,Y,Z\in B$ such that 
	\begin{displaymath}
	X Y = \sum_{Z\in B}N_{XY}^Z Z.
	\end{displaymath}
	\item (Duality) A bijection $*:B\rightarrow B$ such that $1_R^*=1_R$ which extends to an anti-involution on $(R,B)$, i.e. $(XY)^*=X^*Y^*,\forall X,Y\in B$.
\end{itemize}

By associativity the structure constants must satisfy 
\begin{equation}
\sum_{U\in B}N_{XY}^U N_{UZ}^W = \sum_{V\in B}N_{YZ}^V N_{XV}^W
\end{equation}
for all $U,V,W,X,Y,Z\in B$.

The structure constants define a map $N:B^{\times 3}\rightarrow \Z^+$ which we can extend recursively to arbitrary $n+1>3$ via
$$
N_{X_1 \ldots X_n}^{Y}:=\sum_{Z\in B}N_{X_1\ldots X_{n-1}}^{Z}N_{Z X_n}^Y
$$
for all $X_1,\ldots, X_n, Y\in B$.
Those structure constants which are non-zero will play an important role and so we define
$$\Gamma(R,B)=\{(X,Y,Z)\in B^{\times 3}\vert N_{XY}^Z\neq 0\}$$
with $\gamma_{XY}^Z$ the notation for $(X,Y,Z)$. 
We will say that $(R,B)$ is {\em multiplicity free} if $N(\Gamma(R,B))=\{1\}$.

For based rings $(R,B)$ and $(R,B')$, every bijection $\rho:B\rightarrow B'$ satisfying
\begin{equation}
N_{XY}^Z=N_{\phi(X)\phi(Y)}^{\phi(Z)},\forall X,Y,Z\in B
\end{equation}
defines a unique based ring isomorphism $\rho':(R,B)\rightarrow (R,B')$ and vice versa.
Let $Aut(R,B)$ to be the group of based ring automorphisms of $(R,B)$.
\end{definition}
\begin{remark}
Since all based rings in this paper are all multiplicity free, unless otherwise noted all future statements will be made under this assumption.
\end{remark}
\begin{definition}
Let $(R,B)$ be a based ring and define  $N_X$ be the matrix with entry $(N_{X})_{YZ}=N_{XY}^Z$. 
The {\em Frobenius-Perron dimension} $FP(X)$ of $X$ is the largest positive real eigenvalue of $N_X$.
The Frobenius-Perron dimension of $(R,B)$ is $FP((R,B))=\sum_{X \in B} FP(X)^2$.
\end{definition}

We now define fusion categories.

\begin{definition}\label{def:fusioncategory}
A {\em fusion category} $\mathcal C$ over $\C$ is a monoidal semi-simple Abelian category\footnote{We will denote the monoidal bifunctor of a fusion category $\mathcal C$ by $\otimes$ and direct sum of objects by $\oplus$} with identity $\mathbf 1$ such that
\begin{enumerate}
	\item ($\C$-linearity) $\mathcal C$ is enriched over $Vec_{Fin}(\C)$. This is to say that $\mathcal C(a,b)$ is a finite dimensional vector space over $k$ for all objects $a,b\in \mathcal C_0$.
	\item (Finiteness) There are finitely many isomorphism classes of simple objects in $\mathcal C_0$ and $\mathcal C(a,a)\cong \C$ for all simple objects $a\in \mathcal C_0$.
	\item (Rigidity) For every object $a\in \mathcal C_0$, there is an object $a^*\in \mathcal C_0$ and evaluation and co-evaluation maps
	\begin{align*}
	ev_{a}:a \otimes a^*&\rightarrow \mathbf 1 & coev_{a}:\mathbf 1 &\rightarrow a^*\otimes a
	\end{align*}
	such that
	\begin{align}
	\lambda_a \circ (ev_a\otimes Id_a)\circ \alpha_{a,a^*,a}\circ (Id_a \otimes coev_a)\circ \rho_a^{-1}&=Id_a \text{ and }\\
	\rho_{a^*} \circ (Id_{a^*}\otimes ev_{a^*})\circ \alpha_{a^*,a,a^*}^{-1} \circ (coev_{a^*}\otimes Id_{a^*}) \circ \lambda_{a^*}^{-1} &=Id_{a^*}.
	\end{align}
\end{enumerate}
\end{definition}
We denote by $\mathcal K_0(\mathcal C)$ the Grothendieck ring of $\mathcal C$.
This is a based ring as defined in \ref{def:basedring} with basis elements corresponding to equivalence classes of simple objects $a,b,\ldots \in \mathcal C_0$ and multiplication induced from $\otimes$ via $N_{a_1\ldots a_n}^b=dim(\mathcal C(a_1\otimes\ldots \otimes a_n, b))$.
We say that $\mathcal C$ is multiplicity free if $\mathcal K_0(\mathcal C)$ is multiplicity free.
Given two fusion categories $\mathcal C$ and $\mathcal D$ we say that they are {\em Grothendieck equivalent} if and only if $\mathcal K_0(\mathcal C)\cong \mathcal K_0(\mathcal D)$.
We say that a based ring $(R,B)$ {\em admits categorification} if there exists a fusion category $\mathcal C$ such that $(R,B)\cong\mathcal K_0(\mathcal C)$.

\begin{definition}
Let $\mathcal C$ be a fusion category and $\epsilon$ be a natural isomorphism $\epsilon:**\rightarrow Id$ of the double dual and identity functors.
For all equivalence classes of simple objects $a$ of $\mathcal C$ this gives us a morphism $\epsilon_a : a^{**} \rightarrow a$.
From this we can define the left and right {\em quantum trace} of a morphism $f$ on $a$,

\begin{align}
tr_l(f) &=ev_{a}\circ (f \otimes Id_{a^*})\circ ((\epsilon_a)^{-1} \otimes Id_{a^*})\circ coev_{a^*} \label{eq:trl}\\
tr_r(f) &= ev_{a^*}\circ (Id_{a^*}\otimes \epsilon_a)\circ (Id_{a^*}\otimes f)\circ coev_a \label{eq:trr}
\end{align}

where $ev_{\_}$ and $coev_{\_}$ are the evaluation and coevaluation maps.
A natural isomorphism $\epsilon$ such that $\epsilon_{a \otimes b} \cong \epsilon_a \otimes \epsilon_b$ is called a {\em pivotal isomorphism}.
A {\em pivotal fusion category} is a fusion category equipped with a pivotal isomorphism.
\end{definition}

For every isomorphism class of simple objects $a$ in a pivotal fusion category $\mathcal C$, one defines the left quantum dimension $q_l(a)$ and the right quantum dimension $q_r(a)$ as the quantum traces of the identity morphism on $a$.
A {\em spherical fusion category} is a pivotal fusion category such that the left and right quantum dimensions of all objects coincide.
In this case we refer simply to the quantum dimension of an object $a$ as $q_a$.

For a spherical fusion category $\mathcal C$ we define the matrix $D$ to be the diagonal matrix with entries $q_a$ for all equivalence classes $a$. 
We also define the {\em categorical quantum dimension} $q(\mathcal C)= \sum_{a} (q_a)^2$.
A spherical fusion category is called {\em pseudo-unitary} iff $q(\mathcal C)=FP(\mathcal K_0(\mathcal C))$ and there exists a canonical choice of spherical structure such that $q_a = FP(a)$ for all $a$.

There is also the stronger notion of a {\em unitary fusion category}.
\begin{definition}
A {\em conjugation} on a fusion category is a family of conjugate linear maps $\mathcal C(x,y)\rightarrow \mathcal C(y,x)$, $f\rightarrow \bar f$ satisfying
\begin{displaymath}
\begin{array}{lcr}
\bar{\bar f}=f, & \overline{f\otimes g}=\bar f \otimes \bar g, &\text{ and }\overline{f\circ g}=\bar g\circ \bar f.
\end{array}
\end{displaymath}
A fusion category is {\em unitary} if $f=0$ whenever $\bar f\circ f=0$\cite{MR1966524}.
In this case then, all of the $\mathcal C(a\otimes b, c)$ spaces are Hilbert spaces \cite{MR2200691}. Unitary implies $q(a)=FP(a)$ and so then pseudo-unitary and spherical.
\end{definition}

\begin{definition}
Let $\mathcal C$ be a fusion category. 
A braiding on $\mathcal C$ is a family of natural isomorphisms $c_{x,y}: y\otimes x \rightarrow x \otimes y$ satisfying the hexagon equations as in \cite{MR1797619}.
A {\em braided fusion category} is a fusion category equipped with a braiding.
\end{definition}

For braided fusion categories there is a canonical natural isomorphism
\begin{equation}\label{eq:psi}
\psi_a= \rho_{a^{**}} \circ (Id_{a^{**}}\otimes ev_a)\circ \alpha_{a^{**},a,a^*}^{-1}\circ(c_{a,a^{**}}\otimes Id_a^*)\circ\alpha_{a,a^{**},a^*}\circ (Id_a \otimes coev_{a^*})\circ \rho_a^{-1}.
\end{equation}

A {\em balanced fusion category} is a braided pivotal fusion category with balancing $\theta_a= \epsilon_a \circ \psi_a$.
The trace of this morphism, by abuse of notation, will also be denoted $\theta_a$ and the diagonal matrix with entries $\theta_a$ will be called the $T$-matrix.
A {\em ribbon fusion category} is a balanced fusion category that is also spherical.

There also exist a family of invariants $S_{ab}$ corresponding to the evaluations of Hopf links colored by each pair of equivalence classes $(a,b)$.
The matrix with entries $S_{ab}$ is called the $S$-Matrix and a {\em modular category} is a ribbon fusion category with non-degenerate $S$-matrix.
In the case that $\mathcal C$ is modular than all $\mathcal \theta_a$ are roots of unity.
The matrices $\frac{1}{q(\mathcal C)} S$ and $T=diag(\theta_a)$ give a representation of the modular group $SL(2,\Z)$.
It was believed that the pair $(S,T)$ uniquely determine the modular category and typically what one does in classifying modular categories (e.g. \cite{MR2544735},\cite{2013arXiv1310.7050B}) is enumerate the admissible pairs $(S,T)$.
However, it has since then become clear that $(S,T)$ do not suffice to uniquely determine the modular category \cite{mignard}.

\subsection{Fusion and Modular Systems}

An important question is how one goes about constructing fusion categories.
One way to do so\cite{2013arXiv1305.2229D} is through a collection of numbers called a fusion system.

\begin{definition}
A {\em Fusion System} $(L,N,F)(\mathcal C)$ over $k$ consists of
\begin{enumerate}
	\item A set of labels $L$ containing an element called $\mathbf 1$.
	\item An involution $*:L\rightarrow L$ such that $\mathbf 1^* = \mathbf 1$.
	\item A set map $N:L\times L \times L \rightarrow \{0,1\}$ (written $N_{ab}^c$ for $N(a,b,c)$) satisfying
	\begin{align}
		\delta_a^b &= N_{a\mathbf 1}^b=N_{\mathbf 1 a}^b=N_{ab^*}^\mathbf 1=N_{b^*a}^\mathbf 1\label{eq:N1}\\
		N_{abc}^d&:=\sum_{e}N_{ab}^e N_{ec}^d = \sum_f N_{af}^d N_{bc}^f\label{eq:N2} \\
	\intertext{We will define $\Gamma(L,N)=\{\gamma_{ab}^c \vert N(a,b,c)=1\}$.}
	\intertext{\item For every quadruple $a,b,c,d\in L$, an invertible $N_{abc}^d \times N_{abc}^d$ matrix $F_{abc}^d$ with entries satisfying}
		F_{a\mathbf 1 b}^{c;ab} &= N_{ab}^c \label{eq:F1} \\
		F_{a a^{*} a}^{a;1 1} &\neq 0 \label{eq:F2} \\
		\sum_{h} F_{abc}^{h;fg}F_{agd}^{e;hi}F_{bcd}^{i;gj} &= F_{fcd}^{e;hj}F_{abj}^{e;fi}\label{eq:F3}
	\end{align}
\end{enumerate}
The entries of matrix $F_{abc}^{d;ef}$ are referred to as $6J$-symbols.
Define $Aut(N)$ to be the group of permutations on $L$ such that for all $\nu$ in $Aut(N)$ $$N_{ab}^c =N_{\nu(a)\nu(b)}^{\nu(c)}.$$
This then extends to an action on $F$ via 
\begin{align}\label{eq:autaction}
(F_{abc}^{d;ef})^{\nu}:=F_{\nu(a)\nu(b)\nu(c)}^{\nu(d);\nu(e)\nu(f)}.
\end{align}
\end{definition}

Central to our analysis is the ability to interpolate between fusion categories over $\C$ and fusion systems over $\C$.
Given a fusion system $(L,N,F)$ over $\mathcal C$, one can construct a fusion category $\mathcal C(L,N,F)$.
Given a fusion category $\mathcal C$ over $\C$, a fusion system over $\C$, denoted $(L,N,F)(\mathcal C)$, can be extracted such that $\mathcal C((L,N,F)(\mathcal C))\cong \mathcal C$ as fusion categories.
For our construction, this is proven in \cite{2013arXiv1305.2229D} as Proposition 3.7. 
A similar theorem also appears in \cite{MR1659954}.

It is by analyzing a family of fusion systems for $\sost{2p+1}$ Grothendieck rings that we develop our classification.
To do this, we will utilize the following:

\begin{proposition}
Two fusion categories $\mathcal C$ and $\mathcal D$ are Grothendieck equivalent if and only if given any two fusion systems $(L,N,F)$ and $(L',N',F')$ extracted from them there exists a bijection $f:L\rightarrow L'$ such that for all $\nu\in Aut(N)$ there exists $\nu'\in Aut(N')$ such that $f\circ \nu = \nu' \circ f$.
\end{proposition}
\begin{proof}
If $K_0(\mathcal C)\cong K_0(\mathcal D)$ there exists a based ring isomorphism in the sense of \cite{MR1976459} 2.1.iv, and this is determined by its action on basis elements.
\end{proof}

This allows us to specify a common basis for $K_0(\mathcal C)$ from which to work.

There also exists arithmetic data for pivotal structures.
\begin{proposition}
Let $(L,N,F)$ be a fusion system extracted from a fusion category $\mathcal C$ and define the set $\{\epsilon_a\}_{a\in L}$.
Pivotal structures on $\mathcal C$ are in 1-1 correspondence with solutions to
\begin{align}
	\epsilon_{c}^{-1} \epsilon_a \epsilon_b &= F_{abc^*}^{\mathbf 1; c a^*}F_{b c^* a}^{\mathbf 1; a^* b^*}F_{c^* a b}^{\mathbf 1; b^* c}.\label{eq:P1}
\end{align}
\end{proposition}
This is proven in \cite[Proposition 3.12]{2013arXiv1305.2229D}.

Given a fusion category $\mathcal C$ with pivotal structure $\epsilon$, if
it is a spherical structure then there exists a fusion system $(L,N,F)(\mathcal C)$ such that all $\epsilon_a=\pm 1$.
By skeletalizing the rigidity conditions and using the choice of basis as in \cite[Lemma 3.4]{2013arXiv1305.2229D}, the quantum dimensions can be computed from $F$ and $\epsilon$ as 
\begin{align}\label{eq:qdims}
q_l(a) &= \epsilon_a(F_{a^* a a^*}^{a^*; 1 1})^{-1}\text{ and } & q_r(a)&=(\epsilon_a F_{a a^* a}^{a; 1 1})^{-1}
\end{align}

If $\mathcal C$ admits the structure of a unitary fusion category, then there exists a fusion system $(L,N,F)(\mathcal C)$ such that all $F_{abc}^d$ matrices are unitary.
Conversely (See \cite{MR1659954}, Section 4), if given $(L,N)$ there exists a solution to \eq{F1}-\eq{F3} such that all $F_{abc}^d$ matrices are unitary, then $\mathcal C(L,N,F)$ admits a unitary structure.

\begin{definition}
A {\em Modular System} $(L,N,F,R,\epsilon)$ is a fusion system $(L,N,F)$ such that $N_{ba}^c = N_{ab}^c, \forall \gamma_{ab}^c \in \Gamma(L,N)$, $\epsilon=\{\epsilon_a\}_{a\in L}$ is a solution to \eq{P1}, and $\{R_{ab}^c \vert\gamma_{ab}^c \in \Gamma(L,N) \}$ is a collection of numbers satisfying
\begin{align}
	R_{ac}^g F_{acb}^{d; g f} R_{bc}^f &= \sum_{e\in L} F_{cab}^{d; ge} R_{ec}^d F_{abc}^{d; ef} \label{eq:R1} \\
	(R_{ac}^g)^{-1} F_{acb}^{d; g f} (R_{bc}^f)^{-1} &= \sum_{e\in L} F_{cab}^{d; ge} (R_{ec}^d)^{-1} F_{abc}^{d; ef} \label{eq:R2} 
\end{align}
such that the matrix with entries
\begin{equation}\label{eq:smat}
\hat S_{ab} = \sum_{c\in L} F_{ab^* b}^{a; c 1}R_{a b^*}^c R_{b^* a}^c F_{ab^* b}^{a; 1 c}
\end{equation}
is invertible.

\end{definition}

Similar to \ref{eq:autaction}, we can define the action of $\nu\in Aut(N)$ on $R$ via
\begin{align}\label{eq:autactionr}
(R_{ab}^c)^{\nu}:=R_{\nu(a)\nu(b)}^{\nu(c)}.
\end{align}

The matrix $\hat S$ is related to the matrix $S$ via $S = D \hat S D$.
Additionally, we can compute the twists as
\begin{equation}\label{eq:twist}
\theta_a = (q_a)^{-1}\sum_{c\in L} q_c R_{aa}^c
\end{equation}
obtained by taking the trace of the morphism \eq{psi} above\cite{MR2200691}.
Finally, we give a formula for the topological central charge $c_{\rm top}$ (see \cite{MR1104414}).
We introduce
\begin{equation}\label{eq:pplusminus}
p_\pm = \sum_{a\in L} \theta_a^\pm q_a^2 \ ,
\end{equation}
from which one obtains the topological central charge, defined modulo eight,
\begin{align}\label{eq:centralcharge}
 e^{2\pi \imath c/8} &= \frac{p_+}{\sqrt{\sum_a q_a^2}}
&
c_{\rm top} &= c \mod 8 \ .
\end{align}

Note that  through teasing out definitions, all of the information above can be written in terms of $F$, $R$, and $\epsilon$.
Since in our case we have all of the $F$'s, $R$'s, and $\epsilon$'s, so we can simply compute $(S,T)$ for the given the arithmetic data.

\section{Monoidal Equivalence and Gauge Invariants}\label{s:invs}
Given two fusion (modular) categories $\mathcal C$ and $\mathcal D$, a central question is one of whether or not they are equivalent in some suitable sense.
The strongest of these is {\em monoidal equivalence}.
\begin{definition}
$\mathcal C$ and $\mathcal D$ are said to be {\em monoidally equivalent} iff there exists a pair of monoidal functors(see \cite{MR1321145}) $\mathcal F:\mathcal C\rightarrow \mathcal D$ and $\mathcal G:\mathcal C\rightarrow \mathcal D$ such that $\mathcal F \circ \mathcal G$ and $\mathcal G\circ \mathcal F$ are naturally monoidal isomorphic to the identity functors on $\mathcal C$ and $\mathcal D$ respectively.
If $\mathcal C$ and $\mathcal D$ are braided monoidal, then they are {\em braided monoidally equivalent} if $\mathcal F$ and $\mathcal G$ are braided monoidal functors.
\end{definition}

A family of categories $\mathcal C_1,\ldots, \mathcal C_n$ which are monoidally equivalent are automatically Grothendieck equivalent.
A difficult and important question goes the other direction: Given a Grothendieck equivalence class of categories, how many monoidal equivalence classes are there?
That there are finitely many is known as {\em Ocneanu rigidity}\cite{MR2183279}.

Given that we can describe fusion categories arithmetically, there is the natural question of whether or not an equivalence between them can be described arithmetically.
The action $Aut(N)$ on a fusion system $(L,N,F,R,\epsilon)$ is given by \eq{autaction} and it was shown in \cite{2013arXiv1305.2229D} that this gives an object permuting monoidal functor.
We say that two fusion systems $F,F'$ are {\em permutation equivalent} if there exists $\nu \in Aut(N)$ such that $F'=F^\nu$.
The other operation one can perform on $(L,N,F)$ to obtain equivalent categories is given by a {\em gauge transformation}.
This corresponds to a change of basis on the $\mathcal C(a\otimes b,c)$ homspaces and can be represented arithmetically as follows:

Define $G$ to be the set of functions $g:\Gamma(L,N)\rightarrow k^\times$ fixing the notation $g_{ab}^c:= g(a,b,c)$ and $g^{ab}_c:=(g(a,b,c))^{-1}$.
Then given a solution $F$ to \eq{F1}-\eq{F3} and $g\in G$, one obtains another solution $F^g$ via the equations
\begin{align}\label{eq:fgauge}
(F^g)_{abc}^{d;ef}&:=(g_{ab}^e)(g_{ec}^d) F_{abc}^{d;ef} (g^{bc}_f)(g^{af}_d)
\end{align}
and two solutions $F$ and $F'$ are gauge equivalent if and only if there exists $g\in G$ such that $F'=g$.
\eq{fgauge} can be used to construct an object fixing monoidal functor.
Given $\mathcal C$, it is clear from the definition that any two fusion systems extracted from $\mathcal C$ are gauge equivalent.
Monoidal equivalence of fusion categories $\mathcal C$ and $\mathcal D$ is then determined by the existence of a permutation equivalence together with a gauge equivalence.

\begin{remark}
The above extends readily to braided monoidal equivalence between modular categories.
Given $R$ and $R'$, permutation equivalence requires the imposition of the extra condition $R=R^\nu$ as defined in \eq{autactionr}.
Gauge equivalence requires the imposition that $R'=R^g$ as given by
\begin{align}\label{eq:rgauge}
(R^g)&:=(g_{ab}^c) R_{ab}^c (g^{ba}_c).
\end{align}
\end{remark}

Given a set of fusion(modular) systems, determining their monoidal equivalence is not an easy task in general.
However, in \cite{2015arXiv150903275H} a construction is given for invariants which can distinguish gauge classes of fusion categories (with or without multiplicity).
The key to this is noting that for a given $(L,N)$ solutions to the equations \eq{F1}-\eq{F3} define an algebraic scheme $X(L,N)$ in variables $\Phi(L,N)$.
$F$ can then be interpreted as an algebra homomorphism  $F:\C[\Phi(L,N)]\rightarrow \C$ with $F_{abc}^{d;ef}=F(\Phi_{abc}^{d;ef})$ for $\Phi_{abc}^{d;ef}\in \Phi$.
From here, on an open subset $U$ of $X(L,N)$, it is straightforward to extend $F$ to $\mathcal O_X(U)$, the ring of regular functions on $X$ defined at $U$.

Gauge equivalence classes correspond to $G$-orbits in $X(L,N)$ and monoidal equivalence classes correspond to $Aut(N)\ltimes G$-orbits in $X(L,N)$.
Given this, one can leverage the geometric invariant theory of \cite{MR1748380} to obtain the following result:

\begin{theorem}\label{thm:invs}
Fix $(L,N)$.
Let $\mathscr N$ be the number of gauge equivalence classes of fusion categories and $\mathscr M$ be the number monoidal equivalence classes of fusion categories.
Then there exist $P\leq{\mathscr N \choose 2}$ $G$-invariant rational monomials with
\begin{align}\label{eq:invmons}
m_i&=(\phi_{i,1})^{k_{i,1}}\ldots(\phi_{i,j})^{k_{i,j}}, &\text{ with }& &\phi&\in \Phi(L,N) &\text{ and }& &k&\in\Z
\end{align}
such that for fusion systems $F$ and $F'$, $F$ and $F'$ are gauge equivalent if and only if
\begin{align}
F(m_i)&=F'(m_i),& i=1,\ldots P.
\end{align}
Similarly, there exist $Q\leq{\mathscr M\choose 2}$ $Aut(N)$-linear combinations $l_1,\ldots, l_q$ of $G$-invariant monomials as in \eq{invmons} such that $F$ and $F'$ are monoidally equivalent if and only if
\begin{align}
F(l_j)&=F'(l_j),& j=1,\ldots Q.
\end{align}
$Aut(N)$-linear combinations of $G$-invariants will be called $Aut(N)\ltimes G$-invariants for obvious reasons.
\end{theorem}

This is Theorem 1.2 of \cite{2015arXiv150903275H}.
It is straightforward to extend this construction to classifying modular systems $(L,N,F,R,\epsilon)$.
One goes through the same arguments of using the scheme defined from \eq{F1}-\eq{F3},\eq{P1}-\eq{R2}, and noting that they have the same essential properties ($G$ is reductive, orbits are closed).

In the fusion case, the choice of $Aut(N)\ltimes G$-invariants is not necessarily apparent a priori.
However, as we will show in section \ref{s:equivs}, there is an easy choice which works for our $\sost{2p+1}$ categories.

For modular categories there are invariants which immediately present themselves. Even though we previously mentioned that the pair $(S,T)$ is not strong enough for classification \cite{mignard}, $S$ and $T$ prove to be useful.
The quantum dimensions as defined in \ref{eq:qdims} are categorical invariants sufficient to distinguish between spherical structures, but the first row and column of the matrix $ S$ are essentially these numbers.
It has also been previously mentioned that the $q's$, $S's$, and $T's$ can all be written using $F's$, $R's$, and $\epsilon's$.
Thus they all define $G$-invariant regular functions on $X$.

\section{$\sost{2p+1}$ Fusion Systems}\label{s:sost}
In the following we present arithmetic data for the $F$ and $R$ matrices which makes much of the structure for our categories apparent.
We also provide explicit computations of the relevant categorical quantities.

There are several contexts in which the Grothendieck rings for our categories naturally arise, though two primary sources both come from the study of Lie algebras.
As should be apparent from our notation, one of these is as Grothendieck rings for the categories of representations of the (untwisted) affine Kac-Moody algebras $B_p^{(1)}$ (i.e. $\mathfrak{so}(2p+1)$) with highest integral weight 2.
As determined by \cite{MR1384612}, these are equivalent (as modular categories) to the semi-simplification (\`{a} la \cite{MR1182414}) of the representation category for $U_q(B_p)$ with $q=\mathbf e^{\frac{\pi \imath}{2p+1}}$.
The affine Kac-Moody and quantum group constructions for the Grothendieck rings can be found in \cite{MR1887583} and \cite{MR2263097} respectively.

Another source for our fusion rules comes from Tambara-Yamagami categories as follows.
Let $p$ a positive integer, $\chi$ a symmetric bicharacter on $\Z_{2p+1}$ and $\nu=\pm$.
In \cite{MR2587410} it was shown that $\Z_2$-de-equivariantizations of the Tambara-Yamagami category $\mathcal T\mathcal Y(\Z_{2p+1}, \chi, \nu)$ gives rise to an $\sost{2p+1}$ category. 

\subsection{Fusion Rules for $\sost{2p+1}$ categories}\label{ss:FR}
For some $\sost{2p+1}$ category $\mathcal C$, we label the basis elements of $K_0(\mathcal C)$ (i.e. equivalence classes of simple objects) by elements of the set $L=\{\mathbf 1, \epsilon,\phi_i,\psi_{\pm}\}$ with $i=1,\ldots p$.
These have Frobenius-Perron dimensions $\{1,1,2,\sqrt{2p+1}\}$ respectively.
$K_0(\mathcal C)$ is commutative with non-trivial products given by:

\begin{align*}
	\epsilon \otimes \epsilon &\cong \mathbf 1   	& 	\phi_i \otimes \phi_i &\cong \mathbf 1 \oplus \epsilon \oplus \phi_{g(2 i)} 	& 	\psi_{\pm} \otimes \psi_{\pm} &\cong \mathbf 1 \bigoplus_{j=1}^p \phi_j		\\
	\epsilon \otimes \phi_i &\cong \phi_i        	& 	\phi_i \otimes \phi_j &\cong \phi_{g(i-j)} \oplus \phi_{g(i+j)} 	         	&	\psi_{\pm} \otimes \psi_{\mp} &\cong \epsilon \bigoplus_{j=1}^p \phi_j 		\\
	\epsilon \otimes \psi_{\pm} &\cong \psi_{\mp}	&	\phi_i \otimes \psi_\pm &\cong \psi_{\pm} \otimes \psi_{\mp} 
\end{align*}

where $g:\Z_{2p+1} \rightarrow \mathbb G_{2p+1}$ with $\mathbb G_{2p+1} =\{0,\ldots p\}$ is given by $$g(a)=\vert a \vert.$$

\begin{proposition}
Let $K_0(\mathcal C)$ be the Grothendieck ring of a $\sost{2p+1}$ category. Then
\begin{enumerate}
	\item the automorphisms which permute the $\phi_i$ are given by $\mathbb G_{2p+1}^\times := \Z_{2p+1}^{\times}/\langle 1,-1\rangle$,
	\item the automorphisms which permute the $\psi_\pm$ are given by $\Z_2$, 
	\item and the automorphism group of $K_0(\mathcal C)$ is $\mathbb G_{2p+1}^\times \times \Z_2$.
\end{enumerate}
\end{proposition}
\begin{proof}
That the fusion rules are invariant under exchange of $\psi_\pm$ is straight forward.

The automorphisms which permute the $\phi_i$ follow restricting $g$ to $\Z_{2p+1}^\times$ and promoting it to a group homomorphism.
$\Z_{2p+1}^\times$ has even order
\footnote{Specifically the order has to be $\varphi(2 p +1)$ where $\varphi$ is Euler's totient function. 
For power $k$ of prime $r$ $\varphi(r^k)=r^{(k-1)}(r-1)$, which is even. 
That the evaluation of $\varphi$ is even for every $2p+1$ follows from prime factorization and the multiplicative property of $\varphi$.}
and can be represented by integers in $\pm 1,\ldots, \pm p$.
If $i$ is in $\Z_{2p+1}^\times$ so is $-i$, thus the quotient $\Z_{2p+1}^{\times}/\langle 1,-1\rangle$ is well defined.
$g$ is precisely this quotient map.
\footnote{There is not an element of $\mathbb G_{2p+1}^\times$ for all $\pm i$. 
To see this consider $\mathbb Z_9^{\times}$, which has order six rather than eight, since three has no multiplicative inverse in $\mathbb Z_9$.}
$\mathbb G_{2p+1}^\times$ acts on $\phi_i$ in the obvious way and since $g$ is a group homomorphism, this preserves the fusion rules.
There are no automorphisms on pointed objects or between objects of different Frobenius-perron dimension, and so the automorphism group of $K_0(\mathcal C)$ must be the direct product.
This gives a separate proof from that in \cite{MR1887583}.
\end{proof}

\subsection{F-Matrices}\label{ss:FS}

Our solutions to \eq{F1}-\eq{F3} are indexed by pairs $(r,\kappa)$ where $r$ is an odd integer between  $1$ and $2p+1$, such that $gcd(r,2p+1)=1$ and $\kappa=\pm 1$. 

\subsubsection{Notation}
Fix $q=\mathbf e^{\frac{\pi \imath}{2p+1}}$.
To write the general $F$-symbols, we first introduce the following matrices.
\begin{align*}
A(s_1,s_2,s_3,s_4) &= \frac{1}{\sqrt{2}} \begin{pmatrix} s_1 & s_2 \\ s_3 & s_4 \end{pmatrix} &
B &= \begin{pmatrix} 0 & 1 \\ 1 & 0 \end{pmatrix} &
C &= \frac{1}{2} \begin{pmatrix} 1 & -1 & \sqrt{2} \\ -1 & 1 & \sqrt{2} \\ \sqrt{2} & \sqrt{2} & 0 \end{pmatrix} \\
\end{align*}
\begin{align}
D(r;t;s_1,s_2,s_3,s_4) &= \begin{pmatrix}
s_1 Re(q^{r t}) & s_2 Im(q^{r t}) \\
s_3 Im(q^{r t}) & s_4 Re(q^{r t})
\end{pmatrix} \label{eq:dmat}\\
E(r;t;s_1,s_2,s_3,s_4)&= \begin{pmatrix}
s_1 Im(q^{r t}) & s_2 Re(q^{r t}) \\
s_3 Re(q^{r t}) & s_4 Im(q^{r t})
\end{pmatrix}\label{eq:emat}
\end{align}
where the $s_i$ take values $\pm 1$ and satisfy $s_1 s_2 s_3 s_4 = -1$ such that the associated $F$-matrices are orthogonal.

In addition, we define the matrices $G(r,\kappa)$, $H(r,\kappa)$ and $H'(r,\kappa)$ from the following function:
\begin{equation}\label{eq:jfun}
J(i,j;r,\kappa) = \frac{2^{\zeta(i,j)}\kappa}{\sqrt{2p+1}} \left(q^{r}\right)^{ij}
\end{equation}
where $\zeta\left(i,j\right)=\frac{2-\delta_{i0}-\delta_{j0}} 2$.

$G(r,\kappa)$ is a $p\times p$ matrix with entries
\begin{equation}\label{eq:G}
G(r,\kappa)_{i,j} = (-1)^{(i-1)(j-1)} Im\left(J(i,j;r,\kappa)\right)
\end{equation}
whose indices run over $\{1,\ldots, p\}$.
$H(r,\kappa)$ and $H'(r,\kappa)$ are $(p+1)\times(p+1)$ matrices with entries
\begin{align}
H(r,\kappa)_{i,j} &= (-1)^{ij} Re\left(J(i,j;r,\kappa)\right)\label{eq:H}\\
H'(r,\kappa)_{i,j} &= (-1)^{\delta_{i0}+\delta_{j0}+1}H(r,\kappa)_{i,j}\nonumber
\end{align}
whose indices run over $\{0,\ldots, p\}$.
We note that the matrices $G(r)$ and $H^{(')}(r)$ are orthogonal provided that $r$ is an odd integer relatively prime to $2p+1$.

\subsubsection{Arithmetic Data}\label{ss:AD}
Before we give the specific $6J$ symbols, we make a few general remarks. 
First, the order of the entries in the $F$-matrices respects the order we specified above: $(\mathbf 1,\epsilon,\phi_i,\psi_\pm)$, where $i = 1,\ldots,p$.
Second, we will actually not give the values $F_{abc}^{d;ef}$ individually, but instead give the $F$-matrices $F_{abc}^d$.

In the basis we use to give the $F$-matrices, we have the following property
\begin{equation}
\label{eq:rotate}
F_{bcd}^{a} = \bigl( F_{abc}^{d}\bigr)^T \ .
\end{equation}
In addition, it is implicitly assumed that the label $d$ of $F_{abc}^{d}$ is in the tensor decomposition of $a \otimes b \otimes c$.
This allows us to only specify a reduced set of $F$-matrices, while the others can be deduced by using this relation. 
We note that often, but not not always, the $F$-matrices are symmetric. 
In addition, all the $F$-symbols are real in our basis.

We first observe that from \cite{MR3163515} that for fixed $p$ the fusion rules for $\{1,\epsilon, \psi_i\}$ give the tensor structure for $Rep(D_{2p+1})$.
$K_0(\mathcal C)$ is then a $\Z_2$-extension of $K_0(Rep(D_{2p+1}))$.
We will proceed by building subcategories Grothendieck equivalent to $\Z_2$ and $Rep(D_{2p+1})$ before finally specifying $F$-matrices which correspond to the $\sqrt{2p+1}$ objects.

First, if one or more of the labels $a,b,c,d$ equals $\id$, $F_{abc}^{d}=(1)$.
Next we have $\{\id, \epsilon\}$ equivalent to $\Z_2$ and the action of $\epsilon$ on the objects $\phi_i$ is determined by
\begin{align}
F_{\epsilon\epsilon\epsilon}^{\epsilon} &= \left(1\right) &
F_{\epsilon\phi_i\epsilon}^{\phi_i} &= \left(1\right) &
F_{\epsilon\epsilon\phi_i}^{\phi_i} &= \left(-1\right) \label{eq:eel}&
\end{align}
\begin{align}
F_{\epsilon\phi_i\phi_j}^{\phi_k} &= (-1^{j \bmod 2}) && (j\leq i) \label{eq:l1}\\
F_{\epsilon\phi_i\phi_j}^{\phi_k} &= (-1^{j \bmod 2}) && (j > i \wedge k = g(i+j)) \label{eq:l2}\\
F_{\epsilon\phi_i\phi_j}^{\phi_k} &= (-1^{(j-1) \bmod 2}) && (j > i \wedge k = g(i-j))\label{eq:l3}
\end{align}
The rest of the $F$-matrices for the category $Rep(D_{2p+1})$ are given by: 
\begin{align}
F_{\phi_i\phi_i\phi_i}^{\phi_i} &= C \\
F_{\phi_i\phi_j\phi_i}^{\phi_j} &= B && (j \neq i) \\
F_{\phi_i\phi_i\phi_j}^{\phi_j} &= A(1,1,-1^{(j-i+1) \bmod 2},-1^{(j-i)\bmod 2}) && (j\neq i)\\
F_{\phi_i\phi_j\phi_k}^{\phi_l} &= (1) && \text{(O.W.)}\label{eq:ow}
\end{align}

\begin{proposition}
There are no other solutions to the pentagon equations for $Rep(D_{2p+1})$ which extend to solutions for $\sost{2p+1}$.
\end{proposition}
\begin{proof}
By \cite{MR2832261} all fusion categories $\mathcal C$ for $Rep(D_{2p+1})$ are group theoretical and thus Morita equivalent to a pointed fusion category $\mathcal D$ of the form $(D_{2p+1},\omega)$ where $\omega \in H^3(D_{2p+1},\C^\times)$, the group of $2p+1$ roots of unity.
By \cite{MR1976459} we know that (left) module categories over $D_{2p+1}$ are parameterized by $(H,\zeta)$ where $H\leq G$ such that $\omega\vert_H=1$ and $\zeta\in H^2(H,\C^\times)$.

In our case we have that since the universal grading of $\sost{2p+1}$ is $\Z_2$, we have that $\{\psi_+,\psi_-\}$ are indecomposable $\Z_2$ module categories over $Rep(D_{2p+1})$.
We can count which $Rep(D_{2p+1})$ categories have $\Z_2$ module categories by looking at which $D_{2p+1}$ categories have $\Z_2$ module categories.
Since $H^3(D_{2p+1},\C^\times)\cong \Z_{2p+1}$ the only $\omega$ which fixes the $\Z_2$ subgroup of $D_{2p+1}$ is $\omega=1$ and its cohomology is $\Z_2$.
Thus only $Rep(D_{2p+1})$ has indecomposable $\Z_2$ module categories and there are two of them.
\end{proof}

We now specify the $F$-matrices that are labeled by $\psi_\pm$ and start with the $F$-matrices whose labels consist of only $\psi_\pm$.
\begin{align}
F_{\psi_\pm\psi_\pm\psi_\pm}^{\psi_\pm} &= H(r,\kappa) \label{eq:four}\\ 
F_{\psi_\pm,\psi_\mp,\psi_\pm}^{\psi_\mp} &= -H(r,\kappa) \label{eq:twomixed}\\
F_{\psi_\pm\psi_\pm\psi_\mp}^{\psi_\mp} &= H'(r,\kappa) \label{eq:twonext}\\   
F_{\psi_\pm\psi_\pm\psi_\pm}^{\psi_\mp}   &= \pm G(r,\kappa)\label{eq:three}
\end{align}
There are three classes of $F$-matrices involving two labels that are $\psi_\pm$.
The first class of this type is
\begin{align}
F_{\phi_i\psi_\pm\phi_j}^{\psi_\pm} &= \mp(-1)^{ij} D(r;ij;-1,-1^{(i+j)},-1^{(i+j)},1) \label{eq:twotwonext}\\
F_{\phi_i\psi_\pm\phi_j}^{\psi_\mp} &= -(-1)^{ij} E(r;ij;-1^{(i + j)},1,1,-1^{(i+j+1)}) \label{eq:twotwomixed}
\end{align}
where $i,j$ may or may not be equal.
The second class of $F$-matrices with two labels equal to $\psi_\pm$ and $i \neq j$ is
\begin{align*}
F_{\phi_i\phi_i\psi_\pm}^{\psi_\pm} &= A(1,1,1,-1) \\
F_{\phi_i\phi_j\psi_\pm}^{\psi_\pm} &= A(\pm 1,\pm 1,1,-1) && (i-j) \bmod 2 = 0 \\
F_{\phi_i\phi_j\psi_\pm}^{\psi_\pm} &= A(1,-1,\pm 1,\pm 1) && (i-j) \bmod 2 = 1 \\
F_{\phi_i\phi_i\psi_\pm}^{\psi_\mp} &= A(-1,-1,(-1)^{i},(-1)^{(i+1)}) \\
F_{\phi_i\phi_j\psi_+}^{\psi_-} &= A((-1)^{(j+1)},(-1)^{(i+1)},(-1)^{j},(-1)^{(i+1)}) && i < j \\
F_{\phi_i\phi_j\psi_-}^{\psi_+} &= A((-1)^{i},(-1)^{j},(-1)^{i},(-1)^{(j+1)}) && i < j \\
F_{\phi_i\phi_j\psi_+}^{\psi_-} &= A((-1)^{j},(-1)^{i},(-1)^{j},-1^{(i+1)}) && i > j \\
F_{\phi_i\phi_j\psi_-}^{\psi_+} &= A((-1)^{(i+1)},(-1)^{(j+1)},(-1)^{i},-1^{(j+1)}) && i > j \\
\end{align*}
Finally, the third class of $F$-matrices with two labels equal to $\psi_+$ or $\psi_-$ is
\begin{align*}
F_{\epsilon\phi_i\psi_\pm}^{\psi_\pm} &= F_{\epsilon\phi_i\psi_\pm}^{\psi_\mp} = F_{\epsilon\psi_\pm \phi_i}^{\psi_\pm} = F_{\epsilon\psi_\pm \psi_\pm}^{\phi_i} = F_{\epsilon\psi_\pm \psi_\mp}^{\phi_i} =  \left( 1\right) \\ \\
F_{\epsilon\psi_\pm\phi_i}^{\psi_\mp} &= F_{\epsilon\epsilon\psi_\pm}^{\psi_\pm} = F_{\epsilon \psi_\pm \epsilon}^{\psi_\pm} = \left(-1\right) \\
\end{align*}
With these symbols, and the general rules described above, we have exhausted all
the $F$-matrices.

\subsection{$R$-Matrices}
Given our solutions to the pentagon equations, one would anticipate it is not too complicated to construct solutions to the hexagon equations as well.

Given a solution to the hexagon equations one always has a second solution given by inverses.
This simply corresponds to a choice of $R$ and $R^{-1}$ in \eq{R1} and \eq{R2}.
We may also obtain another solution by replacing all $R$-symbols $R_{ab}^{c}$, such that $a$ and $b$ are both one of $\psi_\pm$ by $-R_{ab}^{c}$.
However, this is monoidally equivalent to our original solution via the automorphism $\psi_\pm\rightarrow \psi_\mp$.
This gives us that for any $p$, solutions to the pentagon and hexagon equations can be uniquely identified by the tuple $(p,r,\kappa,\lambda)$, where $\lambda=\pm 1$ indicates whether or not one is referring to the R-symbols given below or their inverses.
We then provide only one such solution.

We specify the $R$-symbols using a form inspired by conformal field theory.
Namely, for each simple object type, we introduce the scaling dimensions $h_a$ which determine the $R$-symbols up to a sign 
\begin{displaymath}
R_{ab}^{c} = (\sigma_{1})_{ab}^{c} (\sigma_{2})_{ab}^{c}(-1)^{h_{a} + h_{b} - h_{c}},
\end{displaymath}
where $(\sigma_{1})_{ab}^{c} = \pm 1$ and $(\sigma_{2})_{ab}^{c} = \pm 1$.
Below we specify the scaling dimensions $h_{a}$ for all the simple objects, as well as the signs $(\sigma_{1})_{ab}^{c}$ and $(\sigma_{2})_{ab}^{c}$.

The scaling dimensions are given by
\begin{align*}
h_\id &= 0 \\
h_\epsilon &= 1 \\
h_{\phi_{i}} &= \frac{ri(2p+1-i)}{2(2p+1)} \\
h_{\psi_+} &= \frac{r(p + \kappa \, s_p - (2p + 1 | r) + 2)}{8} \\
h_{\psi_-} &= \frac{r(p + \kappa \, s_p - (2p + 1 | r) + 6)}{8} \\
\intertext{where }
s_p &= -(-1)^{\frac{p(p+1)}{2}} = -(2|2p+1) \ .
\end{align*}
The notation $(j | n)$ denotes the Jacobi symbol, which is defined for $j$ an integer, and $n$ a positive, odd integer.

To completely specify the $R$-symbols we still need to specify the signs $(\sigma_{1})_{ab}^{c}$ and $(\sigma_{2})_{ab}^{c}$. 

We start with $(\sigma_{1})_{ab}^{c}$, and note that for the $R$-symbols presented here we have the property $(\sigma_{1})_{bc}^{a} = (\sigma_{1})_{ab}^{c}$, which implies that we only need to specify a limited set of these symbols. 
The non-trivial symbols come in two classes.
The first class is
\begin{equation*}
(\sigma_{1})_{\phi_i\psi_\pm}^{\psi_\pm} =
(\sigma_{1})_{\phi_i\psi_\pm}^{\psi_\mp} =
\begin{cases}
-1 & ((i \bmod 4) = (1 \vee 2)) \wedge (p \bmod 2)=1 \\
-1 & ((i \bmod 4) = (2 \vee 3)) \wedge (p \bmod 2)=0 \\
+1 & \text{ otherwise.}
\end{cases}
\end{equation*}
The second class of symbols is
\begin{equation}
(\sigma_{1})_{\phi_i\phi_j}^{\phi_{g(i + j)}} = (-1)^{(i j)} \ .
\end{equation}
The symbols $(\sigma_{1})_{ab}^{c}$ that are not specified by the rules above, are all
equal to $+1$.

Finally, the signs 
$(\sigma_{2})_{ab}^{c}$
also satisfy the property
$(\sigma_{2})_{bc}^{a} = (\sigma_{2})_{ab}^{c}$ and
there is only one class of non-trivial signs of type, namely
\begin{equation*}
(\sigma_{2})_{\epsilon\psi_\pm}^{\psi_\mp} = (\sigma_{2})_{\phi_i\psi_\pm}^{\psi_\mp} = (-1)^{\frac{r-1}{2}} \ .
\end{equation*}
The symbols $(\sigma_{2})_{ab}^{c}$ that are not specified by the rules above, are all equal to $+1$.

With this, we have completely specified one of the solutions of the hexagon equations, from which the other one follows by
taking the inverse.

\begin{remark}
We close by commenting briefly on the appearance of the Jacobi symbols in the expression for the $R$-symbols. 
One can think of the list of hexagon equations as labeled by the different $F$-symbols, namely the one appearing on the left hand side of the symbolic form $R F R = \sum F R F$. 
To derive the form of $h_{\psi_\pm}$ above, consider the hexagon equation associated to $F_{\psi_+\psi_+\psi_+}^{\psi_+;\id\id}$.
The sum on the right hand side gives rise to a quadratic Gauss sum (because one needs to sum over all $\phi_i$ (as well as the identity)), which are closely related to the Jacobi symbols. 
This leads to their presence in the scaling dimensions $h_{\psi_\pm}$. See Appendix~\ref{a:proof} for more details.

As stated above, the Jacobi symbols $( j | n )$ are defined for positive odd integers $n$ and arbitrary integers $j$. 
Interestingly, we observed that the these Jacobi symbols can be written in terms of the matrices $H(r,\kappa)$ and $G(r,\kappa)$ (whose sizes depend on $p$).
In particular, one can show $$( r | 2p +1 ) = \det (H( 2r + 2p+1 )) \det (G(2r+2p+1)).$$
This gives an analytic function of $r$, that goes through all the
Jacobi symbols, defined for $r$ integer.
We note that Eisenstein also constructed such a function,
$(q | p) = \prod_{n=1}^{(p-1)/2} \frac{\sin(2\pi qn/p)}{\sin(2\pi n/p)}$,
see for instance \cite{MR1761696}.
The difference with the function that we found, is that the values $\pm 1$ are the extrema of the function while this is certainly not the case for Eisenstein's function.
\end{remark}

\subsubsection{Modular data for $\sost{2p+1}$ Modular Systems}
All of our F-Matrices are manifestly unitary, thus all of our categories admit at least one spherical structure.
Since they are also braided, they must be ribbon and so we can use the formula \eq{twist}.
Our formula \eq{smat} is defined for any braided category, however one can obtain the more common form by taking $D S D$.
We present the latter to more easily demonstrate that we obtain the same modular result as \cite{MR2587410}.
However, we also provide an explicit classification at the fusion level.

Pivotal structures on modular categories are in bijective correspondence with the group of invertible objects and spherical structures with the maximal Abelian 2-subgroup\cite[Lemma 2.4]{2013arXiv1310.7050B}.
Since the group of invertible objects is $\Z_2$, there are exactly two pivotal structures for each category, both of which are spherical.
It is straightforward to compute that for $\mathbf 1$, $\epsilon$, and $\phi_i$, all pivotal coefficients must be $1$. 
We also have $\epsilon_{\psi_+}=\epsilon_{\psi_-}=\pm 1$, the choice of which switches the sign of $q_{\psi_\pm}$. 

Below, we present the modular data for solutions with positive quantum dimensions.
It's worth noting that for all $p$ and $r$ the pivotal coefficients which yield positive quantum dimensions for $\kappa=1$ are not those that yield positive quantum dimensions for $\kappa=-1$.
This can be seen from noting the effect of $\kappa$ on \eq{qdims}. In particular, $q_{\psi_\pm} > 0$ if we take $\epsilon_{\psi_\pm} = \kappa$.

We then have that, for modular categories corresponding to the solution indexed by $(p,r,\kappa,\lambda)$,
\begin{align}
S_{\mathbf 1 \mathbf 1} 	&= 1 					&S_{\epsilon \epsilon} 		&= 1 								&S_{\phi_i \phi_j} 		&=  4\cos\left(\frac{2 \pi i j r}{2p+1}\right) \label{eq:sm1}\\
S_{\mathbf 1 \epsilon}     	&= 1 					&S_{\epsilon \phi_i}		&= 2 								&S_{\phi_i \psi_\pm}	&= 0 \label{eq:sm2}\\
S_{\mathbf 1 \phi_i} 		&= 2					&S_{\epsilon \psi_\pm}		&= -q_{\psi_\pm}	&S_{\psi_\pm \psi_\pm}	&= -\kappa (2 | 2p+1)q_{\psi_\pm}\label{eq:sm3}\\
S_{\mathbf 1 \psi_{\pm}}	&=q_{\psi_\pm}			&							&									&S_{\psi_\pm \psi_\mp} 	&= \kappa (2 | 2p+1)q_{\psi_\mp} \label{eq:sm4}
\end{align}
For the twists we have
\begin{align}
heta_\epsilon 			&= 1 					&\theta_{\phi_i} 			&= (-1)^i \mathbf{e}^{\imath \pi \frac{\lambda r i^2}{2p+1}} \\
\theta_{\psi_\pm}&= \mp (-1)^{\frac{\kappa \lambda r}{2}} \mathbf{e}^{\frac{\imath \pi \lambda r}{4} ((2p+1|r)+\kappa (2|2p+1)-p)} \ .
\label{eq:tw1} 
\end{align}

For completeness, we also calculate the topological central charge from eqs. \eq{pplusminus},\eq{centralcharge}. We first note that the contributions from $\psi_\pm$ in the sum to obtain $p_+$ cancel. Using similar arguments as given in Appendix~\ref{a:proof}, one finds that the remaining terms constitute a quadratic Gauss sum, resulting in
$p_+/\sqrt{\sum_a q_a^2} = \epsilon_{2p+1} (2\lambda r | 2p+1)$, where $\epsilon_{2p+1} = 1$ for $p$ even, and $\epsilon_{2p+1}=\imath$ for $p$ odd.
Using standard manipulations of the Jacobi symbol, one can finally write
\begin{equation}
c_{\rm top} = 2 p (\lambda + 2 p) -2 + 2 (r | 2p+1) \bmod 8 \ .
\end{equation}

\section{Monoidal Equivalence of $\sost{2p+1}$ Fusion Systems}\label{s:equivs}
In this section we introduce those gauge invariants which we will use to classify monoidally inequivalent solutions for a given $p$ using the method of Section \ref{s:invs}.
At the level of gauge equivalence it is easy to see that two solutions $(r,\kappa)$ and $(r',\kappa')$ are gauge equivalent if and only if $r= r'$ and $\kappa=\kappa'$.
Our goal then is to find a set $\{l_1,\ldots, l_n\}$ of $Aut(N)$-invariant linear combination of $G$-invariant monomials as defined in Theorem \ref{thm:invs} which determine the monoidal classes.
We then determine the number of monoidal equivalence classes for a given $p$.
For all pairs $(r,\kappa)$ as defined in Section \ref{ss:FS} let $F_{(r,k)}$ be the evaluation map as defined for Theorem \ref{thm:invs}.

\subsection{Determining Equivalence}

We first note that the adjoint category is basically useless to us.
\begin{proposition}\label{prop:adjsub}
The monoidal classes of solutions specified in \ref{ss:AD} cannot be distinguished by $Aut(N)\ltimes G$-invariants coming from $Rep(D_{2p+1})$ subcategories.
\end{proposition}
\begin{proof}
For fixed $p$, the only $F$-matrix entries which differ are those in which $r$ arises, i.e. those specified by the $D$, $E$, $H^{(')}$, and $G$ matrices.
These come from $F$-matrices involving at least one $\psi_\pm$ and so do not come from the $Rep(D_{2p+1})$ subcategory.
So if $l$ is an $Aut(N)\ltimes G$-invariant which comes from $Rep(D_{2p+1})$, i.e. only involves the objects $1$, $\epsilon$, and $\phi_i, i=1,\ldots p$ then
$$
F_{(r,k)}(l)=F_{(r',k')}(l).
$$
\end{proof}

From here, we can narrow our search even further.
From \eq{four}-\eq{twotwomixed}, the entries of $$F_{\phi_i\psi_\pm\phi_j}^{\psi_\pm}, F_{\phi_i\psi_\pm\phi_j}^{\psi_\mp},F_{\psi_\pm\psi_\mp\psi_\mp}^{\psi_\mp}, \text{ and }F_{\psi_\pm\psi_\mp\psi_\pm}^{\psi_\mp}$$ can all be written in some fixed way as scalar multiples of entries in $F_{\psi_\pm\psi_\pm\psi_\pm}^{\psi_\pm}$ or $F_{\psi_\pm\psi_\pm\psi_\pm}^{\psi_\mp}$.
These relationships are preserved under permutation.
The entries of $F_{\psi_\pm\psi_\pm\psi_\pm}^{\psi_\pm}$ and $F_{\psi_\pm\psi_\pm\psi_\pm}^{\psi_\mp}$ are given by the matrices $H(r,\kappa)$ and $\pm G(r,\kappa)$ respectively.

\begin{proposition}
Let $F_{(r,\kappa)},F'_{(r',\kappa')}$ be solutions to \eq{F1}-\eq{F3} parameterized by $(r,\kappa)$ and $(r',\kappa')$.
Then $(r,\kappa)=(r',\kappa')$ if and only if $H(r,\kappa)=H(r',\kappa')$.
\end{proposition}
\begin{proof}
The direction $(r,\kappa)=(r',\kappa')\Rightarrow H(r,\kappa)=H(r',\kappa')$ is obvious, so assume $H(r,k)=H(r',k')$.
$H(r,\kappa)$ determines $G(r,\kappa)$ up to a sign and 
\begin{align*}
(F_{(r,k)})_{\psi_\pm\psi_\pm\psi_\pm}^{\psi_\pm} &= (F_{(r',k')})_{\psi_\pm\psi_\pm\psi_\pm}^{\psi_\pm} & &\text{ and }  \\
(F_{(r,k)})_{\psi_\pm\psi_\pm\psi_\pm}^{\psi_\mp}&=\pm (F_{(r',k')})_{\psi_\pm\psi_\pm\psi_\pm}^{\psi_\mp}.
\end{align*}
This corresponds to a choice $\psi_\pm$, which is irrelevant to $(r,\kappa)$, so this implies $(r,\kappa)=(r',\kappa')$.
\end{proof}
\begin{corollary}\label{cor:hdiag}
$(r,\kappa)=(r',\kappa')$ if and only if $$Diag(H(r,\kappa))=Diag(H(r',\kappa')).$$
\end{corollary}
\begin{proof}
$H(r,\kappa)_{0,0}$ determines $\kappa$ and $H(r,\kappa)_{1,1}\propto Re(q^r)$ narrows $r$ to one of two choices.
Taken modulo $2p+1$ only one of $r$ and $-r$ is odd, and so $H(r,\kappa)_{1,1}$ determines $H(r,\kappa)$.
\end{proof}
By inspection, we see that
\begin{align*}
(F_{(r,k)})_{\psi_\pm\psi_\pm\psi_\pm}^{\psi_{\pm;\phi_i\phi_i}}(F_{(r,k)})_{\psi_\pm \phi_i \psi_\pm}^{\phi_i; \psi_\pm \psi_\pm}&=H(r,k)_{i,i}
\intertext{so define }
X_p(r,\kappa,i)&=\frac{\sqrt{2p+1}}{2^{\zeta(i,i)}}\sum_{j\in\{+,-\}} \left( (F_{(r,k)})_{\psi_j\psi_j\psi_j}^{\psi_j;\phi_i\phi_i}(F_{(r,k)})_{\psi_j \phi_i \psi_j}^{\phi_i; \psi_j \psi_j}\right)\\
               &=(-1)^{i^2}\cos\left(\frac{i^2 r \pi}{2p+1}\right). \\
\intertext{and the $(p+1)$-tuple}
X_p(r,\kappa) &=(X_p(r,\kappa,i)\vert i=0,\ldots, p). \nonumber
\end{align*}
\begin{corollary}\label{cor:hinvs}
$X_p(r,\kappa)$ uniquely determines the values for all $G$-invariant monomials and thus all $Aut(N)\ltimes G$-invariant monomials.
\end{corollary}
\begin{corollary}
The permutation $\psi_+\rightarrow \psi_-$ acts as a gauge transformation.
\end{corollary}

There is a natural action of $z\in \mathbb G_{2p+1}^{\times}$ on $X(r,\kappa)$ given by
\begin{align*}
z\cdot X_p(r,\kappa)&=(z\cdot X_p(r,\kappa,i)\vert i=0,\ldots p) \\
             &=(X_p(r,\kappa ,g(z * i))\vert i=0,\ldots, p) \\
			 &=(X_p(r z^2,\kappa,i)\vert i=0,\ldots p) \\
			 &=X_p(r z^2,\kappa).
\end{align*}

\begin{theorem}\label{lem:trace}
Fix $p$ and let $(r,\kappa)$ and $(r',\kappa)$ be solutions as defined in \ref{ss:FS}. 
The following are equivalent:
\begin{enumerate}
	\item $(r,\kappa)$ and $(r',\kappa)$ are monoidally equivalent,
	\item there exists $z\in \mathbb G_{2p+1}^\times$ such that $X_p(r',\kappa) = X_p(r z^2, \kappa)$, and
	\item there exists $z\in \mathbb G_{2p+1}^\times$ such that $g(r') = g(r z^2)$.
\end{enumerate}
\end{theorem}
\begin{proof}
First assume that $(r,\kappa)$ and $(r',\kappa)$ are monoidally equivalent via some $z\in \mathbb G_{2p+1}^\times$.
$z$ acts on $F_{\psi_\pm \psi_\pm \psi_\pm}^{\psi_\pm}$ by 
\begin{align*}
z\cdot F_{\psi_\pm\psi_\pm\psi_\pm}^{\psi_\pm;\phi_i\phi_j}&=F_{\psi_\pm\psi_\pm\psi_\pm}^{\psi_\pm;\phi_{g(z i)}\phi_{g(z j)}}=H(r z^2, \kappa)
\intertext{so}
X_p(r',\kappa)&=X_p(r z^2,\kappa).
\end{align*} 
This implies $r'$ and $r z^2$ are congruent modulo $2p+1$ and so $g(r')=g(r z^2)$.

Assume $\exists z\in \mathbb G_{2p+1}^\times$ such that $g(r') = g(r z^2)$.
Then, modulo $2p+1$, $r' \cong r z^2$ or $r' \cong -r z^2$, but regardless $X_p(r',\kappa)=X_p(r z^2,\kappa)$. 
By Corollary \ref{cor:hinvs}, for any $Aut\ltimes G$-invariant $l$, 
$$F_{(r',\kappa)}(l)=F_{(r,\kappa)}(l).$$ 
Thus $(r',\kappa)$ and $(r,\kappa)$ are monoidally equivalent.
\end{proof}

\subsection{Calculating The Number of Monoidal Equivalence Classes}
It is not an easy task to calculate how many monoidally inequivalent sets of F-moves there are. 
Nevertheless we can construct a framework in which the question is seemingly simple. 
From the previous section we showed that the F-moves labeled by $(r,\kappa)$ and $(r',\kappa')$ are monoidally equivalent if and only if $\kappa=\kappa'$ and there exists a $z\in \mathbb G_{2p+1}^\times$ such that $r'=rz^{2}$. 
Thus we define 
\begin{eqnarray*}
	O_{r} & = & \{rz^{2} | z \in \mathbb G^{\times}_{2p+1}\} \subseteq \mathbb G^{\times}_{2p+1}.
\end{eqnarray*}
Each set contains elements such that the F-moves labeled by $(r,\kappa)$ and $(r',\kappa)$ are monoidally equivalent. Furthermore, $(r,\kappa)$ and $(r',\kappa)$ are monoidally equivalent if and only if $O_{r} = O_{r'}$. We realize that $O_{1}$ is a normal subgroup and that $O_{r}=rO_{1}$ are conjugacy classes. Thus the number of monoidally inequivalent F-moves must be twice (due to $\kappa$) $|\mathbb G^{\times}_{2p+1}|/|O_{1}|$. \\

It is possible to simplify this further by considering the map $f:\mathbb G^{\times}_{2p+1} \rightarrow O_{1}$ defined by $f(z)=z^{2}$. 
This map is clearly onto, i.e. the image of $f$ is $O_{1}$. 
As we are dealing with the finite groups we can use the first isomorphism theorem and deduce that $|\mathbb G^{\times}_{2p+1}|/|O_{1}| = |ker(f)|$ where  $ker(f)$ is the kernel of $f$ and is the subgroup
\begin{eqnarray*}
	ker(f) & = & \{r\in \mathbb G^{\times}_{2p+1} | r^{2}=1\}.
\end{eqnarray*}
Thus the number of monoidally inequivalent F-moves is twice the number of elements of $\mathbb G^{\times}_{2p+1}$ that square to the identity (because of $\kappa$). 
We know that $|ker(f)|=2^{m}$ for some $m\in \mathbb N$.

If $r \in \Z_{2p+1}^{\times}$ satisfies $r^{2}=1$ or $-1$ then $r \in \mathbb G^{\times}_{2p+1}$ satisfies $r^{2}=1$. 
We can then deduce that 
\begin{eqnarray*}
	|ker(f)| & = & \frac{1}{2}\left(|\{r\in \Z_{2p+1}^{\times} | r^{2}=1\}| + |\{r\in \Z_{2p+1}^{\times} | r^{2}=-1\}|\right).
\end{eqnarray*}
Firstly, consider the prime decomposition $2p+1= p_{1}^{a_{1}}...p_{l}^{a_{l}}$ where the $p_{i}$ are prime and $a_{i} > 0$. We define sets
\begin{eqnarray*}
	B_{1} & = & \{r\in \Z_{2p+1}^{\times} | r^{2}=1\}, \\
	B_{2} & = & \{r\in \Z_{2p+1}^{\times} | r^{2}=-1\}.
\end{eqnarray*}
We note that as $B_{1}$ is also a group and that its order is precisely $2^{l}$. Next, suppose that $r,r'\in B_{2}$ it follows that $rr'\in B_{1}$, similarly
if $r\in B_{2}$ and $r'\in B_{1}$ then $rr'\in B_{2}$. Thus we have that the cardinality of $B_{2}$ is zero or $|B_{2}|=|B_{1}|=2^{l}$. For $|B_{2}|$ to be non-zero we must have that there exists $b_{i}$ such that $b_{i}^{2} \equiv -1 \,\mbox{mod}\, p_{i}^{a_{i}}$ for all $i$.\\

\noindent
From this we can deduce the number of monoidal classes
\begin{eqnarray}\label{eq:numcats}
	\# \mbox{monoidal classes} & = & \left\{\begin{array}{cl} 2^{l+1} & \mbox{if}\,\, \exists b \,\,\mbox{s.t.}\,\, b^{2} \equiv -1\,\, \mbox{mod}\,\, (2p+1), \\ 2^{l} & \mbox{otherwise}, \end{array} \right.
\end{eqnarray}
where $l$ is the number of primes that divide $(2p+1)$. Here we have used that $\Z_{2(2p+1)}^{\times} \cong \Z_{2p+1}^{\times}.$

We note that for the cases in which there are $2^l$ monoidal classes, we obtain $2^{l+1}$ modular 
categories, because of the $R$-symbols ($\lambda = +1$) and their inverses ($\lambda = -1$), which are not equivalent.
However, in the cases in which there are $2^{l+1}$ monoidal classes, the modular categories obtained by the $R$-symbols on the one hand and their inverses on the other, are pairwise equivalent.
This means that also in this case, we obtain $2^{l+1}$ modular categories, in accordance with
\cite[Theorem 3.2]{2016arXiv160105460A}.

\section{Examples}\label{s:examples}

In this section we explicitly compute the classifying invariants for a number of ``small'' categories.
The case of $\sost{3}$ is known to be Grothendieck equivalent to $\mathfrak{su}(2)_4$ and both the fusion and modular structures are fully classified (see \cite{MR1237835}).
Beyond this, there are several guideposts:
\begin{itemize}
	\item The classification of weakly integral modular categories of dimension $4m$ is given in \cite{2014arXiv1411.2313B}. This contains those $\sost{2p+1}$ of our family for which $2p+1$ is square free.
	\item The classification of integral modular categories of dimension $4q^2$ is given in \cite{2013arXiv1303.4748B} where $q$ is prime.
	\item Explicit formulae for the modular data of $\Z_2$-equivariantizations of Tambara-Yamagami categories is given in \cite{MR2587410} to which our categories are Grothendieck equivalent.
\end{itemize}
\subsection{$\sost{3}$}
For $\sost{3}$, $\mathbb G_{3}^\times = \{1\}$.
Thus we have only one two dimensional simple object for which the automorphism group is trivial. 
We have two solutions and two fusion categories. 
Up to permutation there are four modular categories, one for each of the four possible T-matrices coming from \eq{tw1}:
\begin{align*}
diag(1,1,-(-1)^{1/3},-(-1)^{1/4},(-1)^{1/4}) & \\
diag(1,1,(-1)^{2/3},(-1)^{3/4},-(-1)^{3/4}) & \\
diag(1,1,-(-1)^{1/3},(-1)^{3/4},-(-1)^{3/4}) & \\
diag(1,1,(-1)^{2/3},-(-1)^{1/4},(-1)^{1/4}) & \\
\end{align*}
Corresponding to solutions $(1,1,1,1)$, $(1,1,1,-1)$, $(1,1,-1,1)$ and $(1,1,-1,-1)$ respectively.
This is in agreement with \cite{MR2587410}.
In this case the two braidings for each fusion category are inequivalent.

\subsection{$\sost{5}$}
There are two two dimensional objects and $\mathbb G_{5}^\times \cong \Z_2$. 
Both elements square to the identity and so there are four fusion categories from \eq{numcats}.
This can be seen from the sets
\begin{align*}
X_2(1,1) &= \left\{1,\frac{1}{4} \left(-1-\sqrt{5}\right),\frac{1}{4} \left(-1-\sqrt{5}\right)\right\}, \text{ and }\\
X_2(3,1) &= \left\{1,\frac{-1}{4} \left(1-\sqrt{5}\right),\frac{-1}{4} \left(1-\sqrt{5}\right)\right\}, \\
\end{align*}

From \cite{MR2587410} there are also four modular categories, corresponding to solutions $(2,1,1,1)$, $(2,1,-1,1)$, $(2,3,1,1)$, and $(2,3,-1,1)$ with $T$-matrices
\begin{align*}
&diag(1, 1, -(-1)^{1/5}, (-1)^{4/5}, -1, 1),\\
&diag(1, 1, -(-1)^{1/5}, (-1)^{4/5}, \imath, -\imath),\\
&diag(1, 1, -(-1)^{3/5}, (-1)^{2/5}, -\imath, \imath), \text{ and}\\
&diag(1, 1, -(-1)^{3/5}, (-1)^{2/5}, 1, -1) .
\end{align*}
In this case the braidings corresponding to $\lambda = \pm 1$ are equivalent.

\subsection{$\sost{7}$}
There are three two dimensional objects and $G_{7}^\times \cong \Z_3$.
We have that
\begin{align*}
X_3(1,1) &= \left\{1,-\cos \left(\frac{\pi }{7}\right),-\sin \left(\frac{\pi }{14}\right),\sin \left(\frac{3 \pi }{14}\right)\right\}, \\
X_3(3,1) &= \left\{1,-\sin \left(\frac{\pi }{14}\right),\sin \left(\frac{3 \pi }{14}\right),-\cos \left(\frac{\pi }{7}\right)\right\},\text{ and} \\
X_3(5,1) &= \left\{1,\sin \left(\frac{3 \pi }{14}\right),-\cos \left(\frac{\pi }{7}\right),-\sin \left(\frac{\pi }{14}\right)\right\}.
\end{align*}

These are all clearly related via permutation.
Computing the modular data for our solutions, we find that there are four modular categories distinguished by their $T$-matrices
\begin{align*}
&diag(1,1,-(-1)^{1/7},(-1)^{4/7},(-1)^{2/7},-(-1)^{1/4},(-1)^{1/4}) \\
&diag(1, 1, (-1)^{6/7}, -(-1)^{3/7}, -(-1)^{5/7}, (-1)^{3/4}, -(-1)^{3/4}) \\
&diag(1,1,-(-1)^{1/7},(-1)^{4/7},(-1)^{2/7},-(-1)^{3/4},(-1)^{3/4}), \text{ and} \\
&diag(1,1,(-1)^{6/7},-(-1)^{3/7},-(-1)^{5/7},(-1)^{1/4},-(-1)^{1/4}) ,
\end{align*}
corresponding to solutions $(3,1,1,1)$, $(3,1,1,-1)$, $(3,1,-1,1)$, and $(3,1,-1,-1)$ respectively.

\subsection{$\sost{9}$}
For $\sost{9}$ there are four two dimensional objects, but we have that $\mathbb G_{9} \cong \mathbb Z_3$ since there are three odd integers less than and coprime to $9$, $1$, $5$, and $7$. 
These correspond to the $S_4$ subgroup generated by the cycle $\langle 124\rangle$. 

The sets $X_4(r,1)$ are then: 

\begin{align*}
X_4(1,1)&=\left\{-\cos\left(\frac \pi 9\right),\cos\left(\frac{4\pi}{9}\right),1,\cos\left(\frac{2\pi}{9}\right)\right\}, \\
X_4(5,1)&=\left\{\cos\left(\frac{4\pi}{9}\right),\cos\left(\frac{2\pi}{9}\right),1,-\cos\left(\frac \pi 9\right)\right\}, \text{ and }\\ 
X_4(7,1)&=\left\{\cos\left(\frac{2\pi}{9}\right),-\cos\left(\frac \pi 9\right),1,-\cos\left(\frac{4\pi}{9}\right)\right\}. 
\end{align*}

From this it's easy to see that $\langle 124\rangle$ sends $X_4(5,1)$ to $X_4(1,1)$ and $\langle 142\rangle$ sends $X_4(7,1)$ to $X_4(1,1)$. 
Thus there is only one equivalence class.
Finally, as before there are four modular structures distinguished by their T-matrices:
\begin{align*}
&diag(1, 1, -(-1)^{1/9}, (-1)^{4/9}, 1, -(-1)^{7/9}, -1, 1), \\
&diag(1, 1, (-1)^{8/9}, -(-1)^{5/9}, 1, (-1)^{2/9}, -1, 1), \\
&diag(1, 1, -(-1)^{1/9}, (-1)^{4/9}, 1, -(-1)^{7/9}, -\imath, \imath),\text{ and} \\
&diag(1, 1, (-1)^{8/9}, -(-1)^{5/9}, 1, (-1)^{2/9}, \imath, -\imath) ,
\end{align*}
corresponding to solutions $(4,1,1,1)$, $(4,1,1,-1)$, $(4,1,-1,1)$, and $(4,1,-1,-1)$ respectively.

\section{Conclusions}

In conclusion, we studied a family of fusion and modular systems with fusion rules that are equivalent to the fusion rules of $\sost{2p+1}$.
The associated anyons are dubbed `metaplectic anyons'.
We obtained explicit expressions for all the $F$- and $R$-symbols (as well as the modular data, that is, the $S$- and $T$-matrices (i.e., the twists) and the central charge).
Based on the explicit expressions for the $F$-symbols, we conjecture a classification of the monoidal equivalence classes via an analysis of suitably chosen gauge invariant combinations of $F$-symbols.
We obtained an expression for the number of monoidal equivalence classes, and finished by giving some explicit examples.
For future work, it would be interesting to see if the work in this paper can be extended to the fusion and modular systems equivalent to the fusion rules of $\sost{2p}$.

\begin{appendices}
\section{Proof of Solution To Pentagon Equations}\label{a:proof}

In this appendix, we briefly explain how we obtained the $F$- and $R$-symbols
that we present in the main text, and how we verified that they indeed are
solutions to the pentagon and hexagon equations, for arbitrary
tuples $(p,r,\kappa,\lambda)$ with $p$ a positive integer, $r$ an odd integer
such that $1\leq r < 2p+1$ and $gcd(r,2p+1)=1$, $\kappa = \pm 1$ and $\lambda = \pm 1$.

We obtained the $F$-symbols, by numerically calculating them from the quantum group based
on $\mathfrak{so}(2p+1)$ at the appropriate roots of unity for small $p$. From this data, we were
able to extract the general pattern, presented in section \ref{ss:FS}.

The main difficulty in verifying that the $F$- and $R$-symbols we obtained do indeed satisfy the
pentagon and hexagon equations lies not so much in the actual verification of the equations,
but rather in convincing oneself that all the pentagon and hexagon equations one needs to verify are
indeed covered.
The reason for this is the structure of the fusion rules of the dimension two objects, which are such
that one has to consider many cases separately. This pertains both to all the possible labelings of
the equations, as well as to the sum present in these equations.

In order to complete this task, we used the algebraic manipulation program {\em Mathematica}.
We first carefully constructed all the pentagon equations that can occur algebraically, that is, for
arbitrary tuples $(p,r,\kappa)$. After having convinced ourselves that we indeed covered all
the possible cases, we then verified this, for low, explicit values of $p$, by using the algebraic
result to explicitly construct the labels of all the pentagon equations for these explicit values of $p$. 
For explicit $p$, one can also generate all possible cases directly. As expected, we indeed covered
all the cases.

After we generated all the cases of the pentagon equations for the tuples $(p,r,\kappa)$, we did
the actual verification. In all but one class of equations, the equations are straightforwardly verified,
and again, we used Mathematica for this. The most complicated
cases occur when the sum over $h$ in
\eq{F1} runs over all dimension two objects (and possibly $\mathbf 1$ or $\epsilon$). To
check that the $F$-symbols satisfy these equations, one rewrites the (maximally triple) products
of sines and cosines that occur as a sum of sines and cosines, and one uses Lagrange equations
to perform the sum. One can actually use the same procedure to verify that the matrices $G$
and $H^{(')}$ \eq{G}, \eq{H} are orthogonal.
 
To verify that the solutions of the $F$- and $R$-symbols presented also satisfy the
hexagon equations \eq{R1},\eq{R2} we used the same strategy as we used to verify the pentagon
equations \eq{F1}. That is, we explicitly constructed all different cases of the hexagon
equations using Mathematica, which we also used to do the actual checking of the equations,
except for one class of equations, that is non-trivial, and was checked by hand.
  
The class of equations that we explicitly verified has all the labels $a,b,c,d$ of \eq{R1},\eq{R2} equal
to $\psi_\pm$. The sum over $e$ runs over all dimension two objects, and possibly $\mathbf{1}$
or $\epsilon$. This class of equations falls apart into four distinct cases, of which we discuss the
simplest case for ease of presentation.
The other three cases can be shown to hold using similar arguments.

We first give the four classes of equations (without simplification) explicitly. The first class is
\begin{align}
\label{eq:H-explicit1}
\nonumber
&\frac{\kappa}{\sqrt{2p+1}} e^{\frac{r \pi \imath}{2}\bigl(2 - \kappa \imath^{p(1+p)} + p - (2p + 1 | r ) \bigr)} =
\\&
\frac{\kappa^2}{2p+1} +  
\frac{2 \kappa^2}{2p+1} \sum_{j=1}^{p} 
e^{\frac{j \pi \imath}{2}  \bigr( j - (-1)^p + r - \frac{j r}{2p+1}\bigr)} .
\end{align}
The second class is  
\begin{align}
\label{eq:H-explicit2}
&\frac{\sqrt{2}\kappa}{\sqrt{2p+1}} 
e^{\frac{\pi \imath}{2}
\bigl(
i_1(i_1-(-1)^p) -r i_1+ \frac{r i_1^2}{2p+1}+
r \bigl( 2 - \kappa \imath^{p(1+p)} + p - (2p + 1 | r ) \bigr)} =\\\nonumber
&\frac{\sqrt{2}\kappa^2}{2p+1} +
\frac{2\sqrt{2} \kappa^2}{2p+1} \sum_{j=1}^{p} (-1)^{j i_1}
\cos(\frac{r i_1 j \pi}{2p+1})
e^{\frac{j \pi \imath}{2}  \bigr( j - (-1)^p + r - \frac{j r}{2p+1}\bigr)} .
\end{align}
The third class of equations reads
\begin{align}
\label{eq:H-explicit3}
&\frac{2\kappa (-1)^{i_1 i_2} \cos(\frac{r i_1 i_2\pi}{2p+1})}{\sqrt{2p+1}} \times \\\nonumber
&e^{\frac{\pi \imath}{2}
\bigl(
i_1(i_1-(-1)^p) -r i_1+ \frac{r i_1^2}{2p+1}+
i_2(i_2-(-1)^p) -r i_2+ \frac{r i_2^2}{2p+1}+
r \bigl( 2 - \kappa \imath^{p(1+p)} + p - (2p + 1 | r ) \bigr)} =\\\nonumber
&\frac{2\kappa^2}{2p+1} + \\\nonumber
&\frac{4 \kappa^2}{2p+1} \sum_{j=1}^{p} (-1)^{j (i_1+i_2)}
\cos(\frac{r i_1 j \pi}{2p+1})
\cos(\frac{r i_2 j \pi}{2p+1})
e^{\frac{j \pi \imath}{2}  \bigr( j - (-1)^p + r - \frac{j r}{2p+1}\bigr)} .
\end{align}
Finally, the fourth class of equations takes the form
\begin{align}
\label{eq:H-explicit4}
&\frac{2\kappa (-1)^{(i_1-1)(i_2-1)} \sin(\frac{r i_1 i_2\pi}{2p+1})}{\sqrt{2p+1}} \times \\\nonumber
&e^{\frac{\pi \imath}{2}
\bigl(
i_1(i_1-(-1)^p) -r i_1+ \frac{r i_1^2}{2p+1}+
i_2(i_2-(-1)^p) -r i_2+ \frac{r i_2^2}{2p+1}+
r \bigl( 2 - \kappa \imath^{p(1+p)} + p - (2p + 1 | r ) \bigr)} =\\\nonumber
&-\frac{4 i \kappa^2 (-1)^{(-i_1-i_2)}}{2p+1}\times \\\nonumber
&\sum_{j=1}^{p} (-1)^{j (i_1+i_2-2)}
\sin(\frac{r i_1 j \pi}{2p+1})
\sin(\frac{r i_2 j \pi}{2p+1})
e^{\frac{j \pi \imath}{2}  \bigr( j - (-1)^p + r - \frac{j r}{2p+1}\bigr)} .
\end{align}
These equations hold if $p$ is a positive integer, $r$ is an odd integer
$1\leq r < 2p+1$ with $gcd(r,2p+1)=1$, $\kappa = \pm 1$ and $i_1, i_2$ are integer. In the
actual hexagon equations one has $i_1,i_2 = 1,2,\ldots,p$.
 
We note that \eq{H-explicit2} (and \eq{H-explicit1}) is obtained from \eq{H-explicit3} by
setting $i_2 = 0$ (and $i_1=0$). We also note that in the right hand side of
\eq{H-explicit4} is a sum of only $p$ terms, in comparison
to \eq{H-explicit1}, \eq{H-explicit2} and \eq{H-explicit3},
for which the right hand side is a sum of $p+1$ terms.

To verify the validity of \eq{H-explicit1}, we consider the RHS of this equation first.
Replacing $j$ by $2p+1-j$ in the summand leaves the summand the same (if $1\leq j \leq p$), which
means we can extend the sum from $j=1,\ldots, p$ to  $j=1,\ldots, 2p$ at the cost of a factor
of two. The sum can then be split in two pieces, namely in the even $j$'s $j = 2l$, with
$l=1,2,\ldots,p$ and odd $j$'s, $j = 2l-(2p+1)$, with $l=p+1,p+2,\ldots,2p$. In both cases, one
checks that the summand takes the form $e^{-\frac{2\pi \imath r l^2}{2p+1}}$, for both $p$ even and
odd. Hence, the RHS becomes a quadratic Gauss sum (see, for instance, \cite{MR1070716})
$$
\frac{1}{2p+1}\sum_{l=0}^{2p} e^{-\frac{2\pi \imath r l^2}{2p+1}} =
\frac{\epsilon_{2p+1}}{\sqrt{2p+1}} (-r | 2p+1) ,
$$
where $\epsilon_{2p+1} = 1$ for $2p+1 = 1 \bmod 4$ (i.e., $p$ even) and
$\epsilon_{2p+1} = \imath$ for $2p+1 = 3 \bmod 4$ (i.e., $p$ odd), and we
used that $\gcd(r,2p+1)=1$.

Standard manipulations of the Jacobi-symbol gives\\
$(-r|2p+1) = (-1)^\frac{(r+1)p}{2} (2p+1 | r)$, so that the RHS takes the form
$$
\frac{\epsilon_{2p+1} (-1)^{\frac{(r+1)p}{2}}}{\sqrt{2p+1}} (2p+1 | r) =
\frac{(2p+1|r)}{\sqrt{2p+1}} \times
\begin{cases}
1 &  p = 0 \bmod2 \\
\imath (-1)^{\frac{r+1}{2}} & p = 1 \bmod2
\end{cases} .
$$ 
We can now compare this result for the RHS of \eq{H-explicit1} with the LHS, by explicitly checking the
LHS for the four different cases of $p$ modulo 4. One finds that the LHS is indeed independent
of $\kappa = \pm 1$, and that the LHS equals the RHS, as we wanted to show.

\end{appendices}


\newcommand{\etalchar}[1]{$^{#1}$}
\def\cprime{$'$}
\providecommand{\bysame}{\leavevmode\hbox to3em{\hrulefill}\thinspace}
\providecommand{\MR}{\relax\ifhmode\unskip\space\fi MR }
\providecommand{\MRhref}[2]{%
  \href{http://www.ams.org/mathscinet-getitem?mr=#1}{#2}
}
\providecommand{\href}[2]{#2}

\end{document}